\documentclass[11pt]{amsart}
\usepackage{amsmath,amsfonts,amsthm,amsopn,amssymb,latexsym,soul}
\usepackage{array}
\usepackage{cite}
\usepackage{color,enumitem,graphicx}
\usepackage[colorlinks=true,urlcolor=blue,
citecolor=red,linkcolor=blue,linktocpage,pdfpagelabels,
bookmarksnumbered,bookmarksopen]{hyperref}
\usepackage[english]{babel}
\usepackage[left=2.61cm,right=2.61cm,top=2.73cm,bottom=2.73cm]{geometry}

\usepackage{bm}

\usepackage{lineno}
\usepackage{epsfig}
\usepackage{graphics}
\usepackage{float}
\usepackage{multirow}
\usepackage{lineno}
\usepackage{fullpage}
\usepackage[normalem]{ulem} 
\usepackage{xspace}
\usepackage{wrapfig}
\usepackage{color}
\usepackage{indentfirst}
\usepackage{comment}

\theoremstyle{definition}
\newtheorem{theorem}{Theorem}
\newtheorem{defi}{Definition}

\newtheorem{ex}{Example}

\numberwithin{defi}{section} 
\numberwithin{theorem}{section}  
\numberwithin{prop}{section}  
\numberwithin{lem}{section}
\numberwithin{equation}{section} 
\numberwithin{rmk}{section} 
\numberwithin{cond}{section}

\allowdisplaybreaks

\title[A comparison of the weakest contractive conditions for Banach and Kannan mappings]{A comparison of the weakest contractive conditions for Banach and Kannan mappings}

\author[S. Hashimoto]{Shunya Hashimoto}
\address{Faculty of Science, Kyoto University,
\newline\indent
Oiwake-tyou Kitashirakawa, Sakyo-ku Kyoto-shi, 606-8502, Japan}
\email{hashimoto.shunya.7m@kyoto-u.ac.jp}

\author[M. Kikkawa]{Misako Kikkawa}
\address{Department of Mathematics, Faculty of Science, Saitama University,
\newline\indent
Shimo-Okubo 255, Sakura-ku Saitama-shi, 338-8570, Japan}
\email{misakokikkawa0731@gmail.com}

\author[S. Machihara]{Shuji Machihara}
\email{machihara@rimath.saitama-u.ac.jp}

\author[A. Saghir]{Aqib Saghir}
\email{saghir.a.460@ms.saitama-u.ac.jp}

\makeatletter
\@namedef{subjclassname@2020}{\textup{2020} Mathematics Subject Classification}
\makeatother

\thanks{
This work was supported by JST CREST, Japan, Grant Number JPMJCR24Q6.
}

\begin{document}
\begin{abstract}
We study the weakest convergence-type conditions for fixed point results for Banach and Kannan mappings. Building on Suzuki\rq{}s weakest condition for Banach mappings and our previous result for Kannan mappings, we compare convergence conditions defined along Picard sequences. We give a direct proof that several weakest convergence conditions are equivalent for Kannan-type mappings on complete metric spaces. 
This proof is achieved without assuming the completeness or the convergence of Picard sequences; it deduces the equivalence only from the existence of fixed points.
In contrast, we construct a counterexample showing that the corresponding equivalence fails for Banach contractions. Finally, we prove that this discrepancy disappears on G-complete metric spaces, clarifying the role of completeness in weakest fixed point theory.
\end{abstract}

\maketitle

\date{}
\maketitle

\section{Introduction}

Fixed point theory is a fundamental branch of nonlinear analysis with wide-ranging applications in differential equations, optimization, and iterative processes. Since the classical Banach contraction principle \cite{Banach}, considerable effort has been devoted to weakening contractive assumptions while preserving the existence, uniqueness, and convergence of fixed points. Among notable alternatives to Banach contractions are Kannan-type mappings, introduced by Kannan in 1968 \cite{Kannan}, whose defining inequality does not require continuity of the mapping. This intrinsic difference has made Kannan contractions an important object of study and has motivated numerous extensions and applications in generalized metric settings and nonlinear frameworks \cite{Khan, Kuczma, Branciari, Regan, Ariza, Roy, Huang}.
Parallel to these developments, a unified theory of contractive mappings emerged through the works of Ćirić \cite{Ciric}, Jachymski \cite{Jachymski}, and Matkowski \cite{Matkowski}, culminating in what is now known as the CJM framework. This approach characterizes contractive behavior via convergence properties of iterates rather than explicit contractive constants and encompasses classical contractions such as Banach, Meir–Keeler \cite{Meir}, Boyd–Wong \cite{Boyd}, and Browder contractions \cite{Browder}.\\
A major breakthrough in identifying weakest possible contractive conditions was achieved by Suzuki. In 2008, Suzuki \cite{Suzuki1} introduced a generalized Banach contraction principle that characterizes the completeness of the underlying metric space. Later, in 2017, Suzuki \cite{S17} established the weakest known contractive conditions for Banach-type mappings in complete metric spaces, showing that the following conditions are equivalent to the existence of a unique fixed point and the convergence of Picard iterates:
\begin{enumerate}
\item[(i)] \(x \neq y \;\Longrightarrow\; d(Tx,Ty) < d(x,y).\)
\item[(ii)] For every \(\varepsilon>0\), there exists \(\delta>0\) such that 
\[d(T^{i}x,T^{j}x)<\varepsilon+\delta
\;\Longrightarrow\;
d(T^{i+1}x,T^{j+1}x)\le \varepsilon,
\quad \forall\, i,j\in\mathbb{N}\cup\{0\}.\]
\end{enumerate}
These results provide a close relationship between contractive behavior and completeness and establish the weakest known condition for Banach contractions in complete metric spaces. Despite the strength of Suzuki\rq{}s results, their scope is primarily confined to Banach-type contractions and does not explicitly address Kannan-type mappings, whose contractive structure is fundamentally different. In our recent work \cite{HKMS25}, we identified the weakest fixed point conditions for Kannan-type contraction on complete metric spaces by applying the CJM framework and restricting the contractive assumptions to Picard sequences. In particular, we showed that a Kannan-type contraction together with the weakest convergence condition along Picard iterates, namely,
\begin{enumerate}
\item[(iii)]For every $\varepsilon>0$, there exists $\delta>0$ such that $\frac{1}{2}\left(d(T^{i}x,T^{i+1}x)+d(T^{j}x,T^{j+1}x)\right)<\varepsilon+\delta\\
\Longrightarrow
d(T^{i+1}x,T^{j+1}x)\le \varepsilon,
\quad \forall\, i,j\in\mathbb{N}\cup\{0\}.$
\end{enumerate}
is equivalent to the existence and uniqueness of fixed points for Kannan-type mappings in complete metric spaces.

In this paper, we compare the weakest conditions for both Banach and Kannan mappings. First, we analyze the equivalence of conditions based on convergence for Kannan-type contraction mappings, providing a direct proof that depends only on the existence of fixed points, rather than on completeness or convergence of iterative sequences to fixed points. This reveals that, unlike for Banach contractions, several seemingly different conditions are equivalent in the Kannan setting. Next, we present a counterexample showing that this equivalence does not hold for Banach contraction mappings. This suggests a fundamental difference between Banach and Kannan mappings in terms of the weakest conditions.
Finally, we prove that by slightly strengthening the notion of completeness—using the concept of G-completeness, originally introduced by Grabiec \cite{G88}, the equivalence of conditions can be recovered for Banach-type mappings. Thus, this work clarifies the precise relationship between Banach and Kannan contractions under the weakest assumptions, highlights the role of completeness in fixed point theory, and extends Suzuki\rq{}s framework to a broader class of nonlinear mappings.

The paper is organized as follows. In Section 2, we collect known weakest fixed point conditions for Banach and Kannan mappings that will be used throughout the paper. Section 3 is devoted to the Kannan case, where we establish the equivalence of several convergence-based weakest conditions by direct arguments. Section 4 presents a counterexample showing that the corresponding equivalence fails for Banach contraction mappings on complete metric spaces. Finally, in Section 5, we introduce G-completeness and show that, under this slightly stronger completeness assumption, the weakest conditions for Banach mappings become equivalent again.

\section{Preliminaries and examples}

In this section, we recall known results on the CJM condition and the weakest fixed point conditions for Banach and Kannan mappings. These results serve as a starting point for the comparison carried out in the subsequent sections.

First, we recall the Banach-type fixed point theorem under the CJM condition introduced by Ćirić \cite{Ciric}, Jachymski \cite{Jachymski}, and Matkowski \cite{Matkowski}.
\begin{theorem}\cite{Ciric, Matkowski, Jachymski} \ 
\label{CJMB}
Let $(X,d)$ be a complete metric space and suppose that the mapping $T: X\to X$ satisfies the following.
\begin{enumerate}
\item[(i)] $x\not=y$ implies $d(Tx,Ty)<d(x,y)$.
\item[(ii)] For every $\varepsilon>0$, there exists $\delta>0$ such that $d(x,y)<\varepsilon+\delta$ implies $d(Tx,Ty)\le \varepsilon$.
\end{enumerate}
Then, $T$ has a unique fixed point.
\end{theorem}

The following theorem, due to Suzuki \cite{S17}, characterizes the weakest fixed point condition for Banach-type mappings on complete metric spaces.
\begin{theorem}\cite{S17} \ 
\label{Bw}
Let $(X,d)$ be a complete metric space and suppose that the mapping $T: X\to X$ satisfies the following.
\begin{align}
\label{CMB}
x\not=y \ \text{implies} \ d(Tx,Ty)<d(x,y).
\tag{CM$_B$}
\end{align}
Then, the following are equivalent.
\begin{enumerate}
\item[(i)] $T$ has a unique fixed point $z\in X$. Moreover, for any $x\in X$, $\{T^nx\}$ converges to $z$.
\item[(ii)] For any $x\in X$ and $\varepsilon>0$, there exists $\delta>0$ such that $d(T^ix,T^jx)<\varepsilon+\delta$ implies $d(T^{i+1}x,T^{j+1}x)\le \varepsilon$ for all $i,j\in\mathbb{N}\cup\{0\}$.
\end{enumerate}
\end{theorem}
The weakest fixed point condition for Kannan-type mappings on complete metric spaces was established recently by the authors using the CJM framework \cite{HKMS25}.
\begin{theorem}\cite{HKMS25} \ 
\label{CJMK}
Let $(X,d)$ be a complete metric space and suppose that the mapping $T: X\to X$ satisfies the following.
\begin{enumerate}
\item[(i)] $x\not=y$ implies $d(Tx,Ty)<\frac{1}{2}\{d(x,Tx)+d(y,Ty)\}$.
\item[(ii)] For every $\varepsilon>0$, there exists $\delta>0$ such that $\frac{1}{2}\{d(x,Tx)+d(y,Ty)\}<\varepsilon+\delta$ implies $d(Tx,Ty)\le \varepsilon$.
\end{enumerate}
Then, $T$ has a unique fixed point.
\end{theorem}
\begin{theorem}\cite{HKMS25} \ 
\label{Kw}
Let $(X,d)$ be a complete metric space and suppose that the mapping $T: X\to X$ satisfies the following.
\begin{align}
\label{CMK}
x\not=y \ \text{implies} \ d(Tx,Ty)<\frac{1}{2}\{d(x,Tx)+d(y,Ty)\}.
\tag{CM$_K$}
\end{align}
Then, the following are equivalent.
\begin{enumerate}
\item[(i)] $T$ has a unique fixed point $z\in X$. Moreover, for any $x\in X$, $\{T^nx\}$ converges to $z$.
\item[(ii)] For any $x\in X$ and $\varepsilon>0$, there exists $\delta>0$ such that \\  
$\frac{1}{2}\{d(T^ix,T^{i+1}x)+d(T^jx,T^{j+1}x)\}<\varepsilon+\delta$ implies $d(T^{i+1}x,T^{j+1}x)\le \varepsilon$ for all $i,j\in\mathbb{N}\cup\{0\}$.
\end{enumerate}
\end{theorem}
To clarify the relationships between the above theorems, we present two examples. The first example demonstrates that the weakest conditions for Banach-type and Kannan-type mappings are independent; neither condition implies the other.
\begin{ex}
\begin{enumerate}
\item[(a)] \textbf{Banach weakest but not Kannan weakest.} \\
Let $X=[0,1]$ with the usual metric $d(x,y)=|x-y|$. Define $T:X \to X$ by $Tx=\frac{x}{2}$. Clearly, $T$ is a Banach contraction so satisfies the weakest condition for a Banach-type. However, $T$ does not satisfy \eqref{CMK}. Indeed, for $x=1$ and $y=0$, we have $d(T1,T0)=\frac{1}{2}$, while $\frac{1}{2}\{d(1,T1)+d(0,T0)\}=\frac{1}{4}$. Thus, \eqref{CMK} does not hold.
\item[(b)] \textbf{Kannan weakest but not Banach weakest.} \\
Let $X = [0,1]$ with the usual metric. Define $T: X \to X$ by
\[ Tx=
\begin{cases}
\frac{x}{4} & \text{if } x \in [0,\frac{1}{2}], \\
\frac{x}{5} & \text{if } x \in (\frac{1}{2},1].
\end{cases}
\]
We consider $x=0.5$ and $y=0.51$. We have $d(x,y)=0.01$, but $d(Tx,Ty)=|0.125-0.102|=0.023>d(x,y)$. Thus, $T$ does not satisfy \eqref{CMB}. 
On the other hand, it can be verified that $T$ satisfies the original Kannan condition with the constant $\alpha=\frac{1}{3}$, that is,
\[ d(Tx,Ty) \le \frac{1}{3}\{d(x,Tx)+d(y,Ty)\} \quad \text{for all} \ x, y \in X. \]
So, $T$ satisfies the weakest condition for a Kannan-type.
\end{enumerate}
\end{ex}
Suzuki \cite[Example 8]{S17} showed that in the Banach-type setting, the weakest condition is strictly weaker than the CJM condition. 
The next example shows that the same phenomenon occurs in the Kannan-type setting.
\begin{ex}
Let $\ell_1=\{a=(a_n)_{n\in\mathbb{N}} : \sum_{n\in\mathbb{N}}|a_n|<\infty\}$ equipped with the standard norm $\|a\|_1=\sum_{n\in\mathbb{N}}|a_n|$, and let $(e_n)_{n=1}^\infty$ be its standard basis. We define the sequences $x_n=(3+\frac{4}{n})e_n$ and $u_n=(1+\frac{1}{n})e_n$ for $n \in \mathbb{N}$. Consider the closed subset $X$ of $\ell_1$ by
\[ X=\{x_n : n \ge 1\} \cup \{u_n : n \ge 1\} \cup \{0\}. \] 
And we define a mapping $T: X \to X$ by 
\[ Tx_n=u_n, \ Tu_n=0, \ \text{and} \ T0=0. \]

First, we show that $T$ satisfies \eqref{CMK}. The most critical case is for $x_n$ and $x_m$ with $n \neq m$. We have:
\begin{align*}
d(Tx_n, Tx_m)&=d(u_n, u_m)=2+\frac{1}{n}+\frac{1}{m}, \\
d(x_n, Tx_n)&=d(x_n, u_n)=2+\frac{3}{n}, \\
d(x_m, Tx_m)&=d(x_m, u_m)=2+\frac{3}{m}.
\end{align*}
It follows that $d(Tx_n, Tx_m)<\frac{1}{2}\{d(x_n, Tx_n)+d(x_m, Tx_m)\}$. 
Strict inequalities also hold for other pairs, so $T$ satisfies \eqref{CMK}. 

Next, we verify that $T$ satisfies the condition (ii) of Theorem \ref{Kw}. Fix $x \in X$ and $\varepsilon>0$. Notice that $T^kx=0$ for all $k\ge 2$, which implies $d(T^kx, T^{k+1}x)=0$ for any $k\ge 2$. When evaluating the implication in condition (ii) for indices $i,j \in \mathbb{N} \cup \{0\}$, we can classify the pairs $(i,j)$ into the following cases:
\begin{enumerate}
\item[(I)] If $i, j \ge 2$, then $d(T^{i+1}x, T^{j+1}x)=d(0,0)=0 \le \varepsilon$ holds for any $\delta>0$.
\item[(II)] If $i \in \{0,1\}$ and $j \ge 2$, then $T^jx=T^{j+1}x=0$. The implication reduces to:
\[ \frac{1}{2}d(T^ix, T^{i+1}x)<\varepsilon+\delta \implies d(T^{i+1}x, 0) \le \varepsilon. \]
Since $T$ satisfies \eqref{CMK} and $T^i x \neq 0$ (unless $x=0$), we have the strict inequality $d(T^{i+1} x, 0)<\frac{1}{2}d(T^ix, T^{i+1}x)$. Thus, the gap 
\[ \delta_1=\min_{i\in \{0,1\}}\left(\frac{1}{2}d(T^i x, T^{i+1}x)-d(T^{i+1}x, 0)\right), \]
is strictly positive.
\item[(III)] If $i,j \in \{0,1\}$, by the same reasoning using \eqref{CMK}, we can define strictly positive gaps 
\[ \delta_2=\min_{i,j\in \{0,1\}}\left(\frac{1}{2}\{d(T^ix, T^{i+1}x)+d(T^jx, T^{j+1}x)\}-d(T^{i+1}x, T^{j+1}x)\right)>0. \]
\end{enumerate}
So, we can choose $\delta>0$ such that $\delta \le \min \{\delta_1, \delta_2\}$. Therefore, condition (ii) of Theorem \ref{Kw} is satisfied for each $x \in X$.

However, $T$ does not satisfy condition (ii) of Theorem \ref{CJMK}. Let $\varepsilon=2$. For any given $\delta>0$, we can choose sufficiently large integers $n, m$ such that $\frac{3}{2n}+\frac{3}{2m}<\delta$. Then, we have
\[ \frac{1}{2}\{d(x_n, Tx_n)+d(x_m, Tx_m)\}=2+\frac{3}{2n}+\frac{3}{2m}<\varepsilon+\delta. \]
But we obtain
\[ d(Tx_n, Tx_m)=2+\frac{1}{n}+\frac{1}{m}>2=\varepsilon. \]
Therefore, $T$ is the weakest mapping for a Kannan-type but is not the CJM mapping for a Kannan-type.
\end{ex}
Theorem \ref{Bw} and Theorem \ref{Kw} provide benchmark weakest conditions for Banach and Kannan mappings, respectively. In the subsequent sections, we investigate to what extent convergence-based conditions similar to (ii) above are equivalent in each setting, and how this equivalence depends on the underlying assumptions.

\section{Kannan case}

We recall an equivalence condition established in \cite{HKMS25}.
Under condition \eqref{CMK}, the sequence $(d(T^nx,T^{n+1}x))_{n\in\mathbb{N}}$ is strictly decreasing, and hence its nonnegative limit exists.
Let us denote this limit by $\alpha\ge0$.
\begin{theorem}
\label{Ke}
Assume the \eqref{CMK} condition of Theorem \ref{Kw}. 
Let $\alpha$ be the limit of the sequence $(d(T^nx,T^{n+1}x))_{n\in\mathbb{N}}$. Then the following statements are equivalent:
\begin{enumerate}
\item[(i)] $\alpha=0.$
\item[(ii)] For any $x\in X$ and $\varepsilon>0$, there exists $\delta>0$ such that 

\noindent
$\frac{1}{2}\{ d(T^ix,T^{i+1}x)+d(T^{i+1}x,T^{i+2}x)\}<\varepsilon+\delta$ implies $d(T^{i+1}x,T^{i+2}x)\le \varepsilon$ for all $i\in\mathbb{N}\cup\{0\}$.
\item[(iii)] For any $x\in X$ and $\varepsilon\in (0,\infty)\backslash \{d(T^kx,T^{k+1}x) : k=1,2,\cdots\}$, there exists $\delta>0$ such that 

\noindent
$\frac{1}{2}\{ d(T^ix,T^{i+1}x)+d(T^{j}x,T^{j+1}x)\}<\varepsilon+\delta$ implies $d(T^{i+1}x,T^{j+1}x)\le \varepsilon$ for all $i,j\in\mathbb{N}\cup\{0\}$.
\end{enumerate}
Moreover, if $T$ has a fixed point $z$, then statement (iv) is also equivalent to (i)-(iii).
\begin{enumerate}
\item[(iv)] For any $x\in X$ and $\varepsilon>0$, there exists $\delta>0$ such that 

\noindent
$\frac{1}{2}\{ d(T^ix,T^{i+1}x)+d(T^{j}x,T^{j+1}x)\}<\varepsilon+\delta$ implies $d(T^{i+1}x,T^{j+1}x)\le \varepsilon$ for all $i,j\in\mathbb{N}\cup\{0\}$.
\end{enumerate}
\end{theorem}
In \cite{HKMS25}, Theorem \ref{Ke} was proved via Theorem \ref{Kw} under the assumption of completeness. Here, we provide a direct proof that is achieved without assuming the completeness or the convergence of Picard sequences; it deduces the equivalence only from the existence of fixed points.
\begin{proof}[Proof of Theorem \ref{Ke}]
The equivalence (i)$\Leftrightarrow$(ii)$\Leftrightarrow$(iii) is proved directly in Proposition 5.1 of \cite{HKMS25}. The implication (iv)$\Rightarrow$(iii) is obvious. Hence, it suffices to show (i)$\Rightarrow$(iv).

By the equivalence (i)$\Leftrightarrow$(iii), for any $k\in\mathbb{N}\cup\{0\}$ it is enough to verify (iv) with
\[ \varepsilon=d(T^kx,T^{k+1}x)>0. \]

By taking the contrapositive of (iv), it suffices to show the following.
\begin{align}
\label{cl}
\inf_{(i,j)\in E}\frac{1}{2}\{d(T^ix, T^{i+1}x)+d(T^jx, T^{j+1}x)\}>\varepsilon,
\end{align}
where $E:=\{(i,j): d(T^{i+1}x, T^{j+1}x)>\varepsilon\}$.

\noindent
Case 1: $i,j\ge k$. By (i) and \eqref{CMK}, the sequence $(d(T^nx,T^{n+1}x))_{n\in\mathbb{N}}$ is strictly decreasing. Then, for any $i,j\ge k$, we have
\begin{align*}
d(T^{i+1}x,T^{j+1}x)&\le \frac{1}{2}\{d(T^ix,T^{i+1}x)+d(T^jx,T^{j+1}x)\} \\
&\le d(T^kx,T^{k+1}x)=\varepsilon.
\end{align*}
Here we have used the fact that, by \eqref{CMK}, for any $x,y\in X$,
\[ d(Tx,Ty)\le \frac{1}{2}\{d(x,Tx)+d(y,Ty)\}. \]
Therefore, $E$ is empty, so \eqref{cl} holds.

\noindent
Case 2: $i,j\le k$. For any $(i,j)\in E$, by \eqref{CMK}, we have
\[ \frac{1}{2}\{d(T^ix,T^{i+1}x)+d(T^jx,T^{j+1}x)\}\ge d(T^{i+1}x,T^{j+1}x)>\varepsilon. \]
Then, since $E$ is a finite set, we obtain
\[ \min_{(i,j)\in E}\frac{1}{2}\{d(T^ix,T^{i+1}x)+d(T^jx,T^{j+1}x)\}>\varepsilon. \]
So \eqref{cl} holds.

\noindent
Case 3: $i<k<j$. 
If $T^ix=z$, then for any $j\ge i$, $T^jx=z$ holds, so (iv) is satisfied. Therefore, we assume $T^ix\not=z$. In this case, from \eqref{CMK}, we have
\begin{align}
\label{32}
d(T^{i+1}x,z)<\frac{1}{2}d(T^ix, T^{i+1}x). 
\end{align}
We distinguish two subcases:

\noindent
Subcase 3.1: $\frac{1}{2}d(T^ix,T^{i+1}x)\le \varepsilon$. 

First, note that using the \eqref{CMK} condition with the fixed point $z$ and assumption (i), we have $d(T^{\ell+1}x,z)<\frac{1}{2}d(T^{\ell}x,T^{\ell+1}x)$ for all $\ell\in\mathbb{N}$, which implies that $\lim_{\ell\to \infty}d(T^{\ell+1}x,z)=0$.
Furthermore, from \eqref{32} and the condition of this subcase, we obtain a strict inequality:
\[ \lim_{j\to\infty}d(T^{i+1}x,T^{j+1}x)=d(T^{i+1}x,z)<\frac{1}{2}d(T^ix,T^{i+1}x)\le \varepsilon. \]
Therefore, it follows that $d(T^{i+1}x,T^{j+1}x)\le \varepsilon$ for all sufficiently large $j$.
Because $k$ is fixed and $i<k$, there are only finitely many choices for $i$. For each such $i$, the inequality $d(T^{i+1}x,T^{j+1}x)>\varepsilon$ can hold for at most finitely many $j$. This implies that the set of pairs $(i,j)\in E$ in this subcase is finite. Hence, analogous to Case 2, the condition \eqref{cl} is satisfied.

\noindent
Subcase 3.2: $\frac{1}{2}d(T^ix,T^{i+1}x)>\varepsilon$. 

In this case, we have
\[ \frac{1}{2}\{d(T^ix, T^{i+1}x)+d(T^jx, T^{j+1}x)\}\ge \frac{1}{2}d(T^ix, T^{i+1}x)>\varepsilon. \]
Thus, we choose $\delta_1>0$ such that
\[ \min_{\substack{i<k, \\ \frac{1}{2}d(T^ix,T^{i+1}x)>\varepsilon}}\frac{1}{2}d(T^ix, T^{i+1}x)\ge \varepsilon+\delta_1>\varepsilon, \]
the assumption cannot be satisfied, and therefore, condition (iv) is verified.

Combining these results, since \eqref{cl} holds in all cases, there exists $\delta_2>0$ such that
\[ \inf_{(i,j)\in E}\frac{1}{2}\{d(T^ix, T^{i+1}x)+d(T^jx, T^{j+1}x)\}\ge \varepsilon+\delta_2>\varepsilon. \]
By setting $\delta=\min\{\delta_1,\delta_2\}$, we have shown (iv).
\end{proof}

\section{Banach case}

Next, we consider the case of Banach contraction mappings.

Let $(X,d)$ be a complete metric space. Assume that $T:X\to X$ satisfies the \eqref{CMB} condition of Theorem \ref{Bw}.
Then we investigate whether the following two conditions are equivalent:
\begin{enumerate}
\item[(B1)] For any $x\in X$ and $\varepsilon>0$, there exists $\delta>0$ such that $d(T^ix,T^jx)<\varepsilon+\delta$ implies $d(T^{i+1}x,T^{j+1}x)\le \varepsilon$ for all $i,j\in\mathbb{N}\cup\{0\}$.
\item[(B2)] For any $x\in X$ and $\varepsilon>0$, there exists $\delta>0$ such that $d(T^ix,T^{i+1}x)<\varepsilon+\delta$ implies $d(T^{i+1}x,T^{i+2}x)\le \varepsilon$ for all $i\in\mathbb{N}\cup\{0\}$.
\end{enumerate}

For Kannan mappings, the corresponding conditions (ii) and (iv) in Theorem \ref{Ke} are equivalent. However, in the Banach case, this equivalence (namely, (B2) $\Rightarrow$ (B1)) does not hold in general. We present the following example.
\begin{ex}
Let $\displaystyle a_n=\sum_{j=1}^n\frac{1}{j}, \ a_0=0$. 
We define $X=\{a_n: n\in\mathbb{N}\cup\{0\}\}\subset \mathbb{R}$, and let $d$ denote the usual metric on $\mathbb{R}$.

Then, $(X,d)$ is complete. Here, we define $T:X\to X$ by $Ta_n=a_{n+1}, \ n\in\mathbb{N}\cup\{0\}$.
Therefore, we can show that $T$ satisfies the \eqref{CMB} condition of Theorem \ref{Bw} and (B2), but does not satisfy (B1).

(\eqref{CMB}) For any $x,y\in X$ with $x\not=y$, we can express $x=T^ma_0$ and $y=T^na_0$ with $m\not=n$. Assume $m>n$. Then, we have
\begin{align*}
d(Tx,Ty)&=d(T^{m+1}a_0, T^{n+1}a_0)=d\left( \sum_{j=1}^{m+1}\frac{1}{j}, \sum_{j=1}^{n+1}\frac{1}{j}\right)=\sum_{j=n+2}^{m+1}\frac{1}{j} \\
&=\sum_{j=n+1}^m\frac{1}{j}+\frac{1}{m+1}-\frac{1}{n+1}<\sum_{j=n+1}^{m}\frac{1}{j}=d(a_n,a_m)=d(x,y).
\end{align*}
Therefore, \eqref{CMB} holds.

((B1) is not satisfied) $T$ has no fixed point by definition. Therefore, by Theorem \ref{Bw}, (B1) does not hold.

((B2)) Without loss of generality, let $x=a_0$. First, we note that
\[ d(T^ia_0,T^{i+1}a_0)=\frac{1}{i+1}, \quad d(T^{i+1}a_0,T^{i+2}a_0)=\frac{1}{i+2}. \]
Therefore, if $\varepsilon\ge 1$, (B2) holds obviously.
For any $\varepsilon\in(0,1)$, there exists $I\in \mathbb{N}\cup\{0\}$ such that $\frac{1}{I+1}\ge \varepsilon$ and $\frac{1}{I+2}<\varepsilon$.
Then, we choose $\delta>0$ such that $\varepsilon\le \frac{1}{I+1}<\varepsilon+\delta<\frac{1}{I}$ where we set $\frac{1}{0}=\infty$.
Then, for any $i\in\mathbb{N}\cup\{0\}$, if $i\ge I$, then $\frac{1}{i+2}<\varepsilon$, so (B2) holds. Moreover, if $i<I$, then $\varepsilon+\delta\le \frac{1}{I}\le \frac{1}{i+1}$, so the assumption of (B2) is not satisfied, so (B2) holds.
\end{ex}

\section{Banach case on G-complete space}

In this section, we consider under what assumptions the two conditions in Section 3 become equivalent.
First, we introduce G-completeness, introduced in \cite{G88}, which is often used in fixed point theory on fuzzy metric spaces.
\begin{defi}
Let $(X,d)$ be a metric space and $\{x_n\}$ be a sequence in $X$.
\begin{enumerate}
\item The sequence $\{x_n\}$ is said to be a G-Cauchy sequence if $\displaystyle \lim_{n\to\infty}d(x_n,x_{n+p})=0$ for $p\in\mathbb{N}$.
\item A metric space in which every G-Cauchy sequence is convergent is called a G-complete metric space.
\end{enumerate}
\end{defi}
Using this definition, we can prove the following theorem.
\begin{theorem}
Let $(X,d)$ be a G-complete metric space and suppose that the mapping $T: X\to X$ satisfies the \eqref{CMB} condition of Theorem \ref{Bw}.
Then, the following are equivalent.
\begin{enumerate}
\item[(i)] $T$ has a unique fixed point $z\in X$. Moreover, for any $x\in X$, $\{T^nx\}$ converges to $z$.
\item[(ii)] For any $x\in X$ and $\varepsilon>0$, there exists $\delta>0$ such that $d(T^ix,T^jx)<\varepsilon+\delta$ implies $d(T^{i+1}x,T^{j+1}x)\le \varepsilon$ for all $i,j\in\mathbb{N}\cup\{0\}$.
\item[(iii)] For any $x\in X$ and $\varepsilon>0$, there exists $\delta>0$ such that $d(T^ix,T^{i+1}x)<\varepsilon+\delta$ implies $d(T^{i+1}x,T^{i+2}x)\le \varepsilon$ for all $i\in\mathbb{N}\cup\{0\}$.
\end{enumerate}
\end{theorem}
\begin{proof}
(i)$\Rightarrow$(ii) follows from Theorem \ref{Bw}; (ii)$\Rightarrow$(iii) is obvious.
Therefore, it suffices to show (iii)$\Rightarrow$(i).

For any $x\in X$, we define $x_{n+1}=T^nx$, where $n\in\mathbb{N}\cup\{0\}$. First, if there exists some $N\in\mathbb{N}$ such that $x_N=x_{N+1}$, then $x_N$ is a fixed point of $T$. Therefore, for any $n\in\mathbb{N}$, we assume $x_n\not=x_{n+1}$. Then from \eqref{CMB}, we have
\[ d(x_{n+1},x_{n+2})=d(Tx_n,Tx_{n+1})<d(x_n,x_{n+1}). \]
Therefore, $\{d(x_n,x_{n+1})\}$ is a strictly decreasing sequence. Therefore, there exists some $\alpha\ge 0$ such that
\begin{align}
\label{eq2}
d(x_n,x_{n+1})\to \alpha, \quad (n\to\infty).
\end{align}
Assume $\alpha>0$. Then, for $\varepsilon=\alpha$, we take $\delta>0$ such that (iii) holds. Then, from \eqref{eq2}, there exists some $N\in\mathbb{N}$ such that $d(x_n,x_{n+1})<\varepsilon+\delta$ for any $n\ge N$. 
Hence, by (iii), we have $d(x_{n+1},x_{n+2})\le \varepsilon=\alpha$.
This contradicts \eqref{eq2}and the strict monotonicity of $\{d(x_n,x_{n+1})\}$.
Thus, $\alpha=0$, that is,
\begin{align}
\label{eq3}
\lim_{n\to\infty}d(x_n,x_{n+1})=0.
\end{align}
Next, we will show that $\{x_n\}$ is a G-Cauchy sequence. Fix $p\ge1$. Then, we have
\[ d(x_n,x_{n+p})\le d(x_n,x_{n+1})+\cdots+d(x_{n+p-1},x_{n+p})\to 0, \ (n\to\infty). \]
Therefore, since $\{x_n\}$ is a G-Cauchy sequence, there exists some $z\in X$ such that $\{x_n\}$ converges to $z$.

Next, we will show that $z$ is a fixed point of $T$. 
By the contractive condition \eqref{CMB}, we have
\[ d(z,Tz)\le d(z,x_{n+1})+d(Tx_n,Tz)\le d(z,x_{n+1})+d(x_n,z)\to 0, \ (n\to\infty). \]
Therefore, $Tz=z$ holds.

Finally, we show the uniqueness of the fixed point. Let $w$ be another fixed point of $T$. By the contractive condition \eqref{CMB}, we have:
\[ d(w,z)=d(Tw,Tz)<d(w,z) \]
This is a contradiction. Thus, $w=z$, which completes the proof of the theorem.
\end{proof}

\section*{Declarations}

\subsection*{Conflicts of interests}
The authors declare that there is no conflict of interest regarding the publication of this paper.

\subsection*{Data Availability Statements}
Data sharing is not applicable to this article as no datasets were generated or analyzed during the current study.

\end{document}